\documentclass[12pt]{amsart}

\usepackage{tikz-cd}
\usepackage[export]{adjustbox}
\usepackage{amsrefs}
\usepackage{amssymb}
\usepackage{fancyhdr}
\usepackage{bbm}
\usepackage{xcolor}
\usepackage{hyperref}
\hypersetup{
  colorlinks=true,
  linkcolor=blue,
  citecolor=red
}

\newtheorem{theorem}{Theorem}
\newtheorem*{theorem*}{Theorem}
\newtheorem{lemma}[theorem]{Lemma}
\newtheorem{proposition}[theorem]{Proposition}
\newtheorem{claim}[theorem]{Claim}

\newtheorem{maintheorem}{Theorem}

\theoremstyle{definition}
\newtheorem{definition}[theorem]{Definition}
\newtheorem*{definition*}{Definition}
\newtheorem*{lemma*}{Lemma}
\usepackage{graphicx}
\usepackage{pstricks, enumerate, pst-node, pst-text, pst-plot}

\numberwithin{equation}{section}
\numberwithin{theorem}{section}

\newcommand{\N}{\mathbb{N}}

\usepackage{xparse}
\DeclareDocumentCommand\Pr{ m g }{%
    {   \IfNoValueTF {#2}
      {\mathbb{P}\left[{#1}\right]}
      {\mathbb{P}\left[{#1}\middle\vert{#2}\right]}%
    }
}
\DeclareDocumentCommand\E{ m g }{%
    {   \IfNoValueTF {#2}
      {\mathbb{E}\left[{#1}\right]}
      {\mathbb{E}\left[{#1}\middle\vert{#2}\right]}%
    }
}

\usepackage{mathtools}

\newcommand{\cent}[1]{C_G(#1)}

\begin{document}

\title[]{Non-virtually nilpotent groups have infinite
  conjugacy class quotients}

\author[J.\ Frisch]{Joshua Frisch}
\author[P.\ Vahidi Ferdowsi]{Pooya Vahidi Ferdowsi}

\address{California Institute of Technology}


\date{\today}

\begin{abstract}
  We offer in this note a self-contained proof of the fact that a
  finitely generated group is not virtually nilpotent if and only if
  it has a quotient with the infinite conjugacy class (ICC)
  propoerty. This proof is a modern presentation of the original
  proof, by McLain (1956) and Duguid and McLain (1956).
\end{abstract}

\maketitle
\section{Introduction}
It has been long known that a finitely generated group is not
virtually nilpotent if and only if it has a quotient with the infinite
conjugacy class (ICC) propoerty. This result, due to
McLain~\cite{mclain1956remarks} and Duguid and
McLain~\cite{duguid1956fc}, has recently found novel applications in
the study of proximal actions and strong
amenability~\cite{frisch2018strong} and the Poisson
boundary~\cite{frisch2018choquet}. This has motivated us to offer in
this note a self-contained presentation of the original proof of this
fact, which is divided between~\cite[Theorem 2]{mclain1956remarks}
and~\cite[Theorem 2]{duguid1956fc}.

Let $G$ be a group. An element $g\in G$ is said to be a {\em finite conjugacy (or FC) element} if it has only finitely many conjugates in $G$. The {\em FC-center} of $G$ is the set of all FC-elements in $G$. The {\em upper FC-series} of $G$ is defined as follows
$$\{e\}=F_0\leqslant F_1\leqslant \cdots \leqslant F_\alpha \leqslant \cdots,$$
where $F_{\alpha+1}/F_\alpha$ is the set of all FC-elements of $G/F_\alpha$, and $F_\beta = \cup_{\alpha<\beta} F_\alpha$ for a limit ordinal $\beta$. This series will stabilize at some ordinal $\gamma$. $F_{\gamma}$ is called the {\em hyper-FC center} of $G$ and the least such $\gamma$ is called the {\em FC-rank} of $G$. If $G$ is equal to its hyper-FC center, $G$ is called {\em hyper-FC central}. 

\begin{maintheorem}
\label{thm: f.g. virtually nilpotent}
  For a finitely generated group $G$ the following are equivalent.
  \begin{enumerate}
    \item $G$ is virtually nilpotent.
    \item $G$ is hyper-FC central.
    \item $G$ has no non-trivial ICC quotients.
  \end{enumerate}
\end{maintheorem}

As a corollary, a finitely generated group is either virtually nilpotent or has an ICC quotient.

\section{Proof}
The following easy, but important, proposition shows that the obstruction to a group being ICC is the hyper-FC center of the group. Before we state the proposition, we define {\em a universal ICC quotient} of a group. This notion is useful to see the relation between the hyper-FC center of a group and its ICC quotients.

\begin{definition}
  Let $G$ be a group. A universal ICC quotient of $G$, which we denote by by $\phi: G\to I$, is a quotient of $G$ onto an ICC group $I$ such that any quotient $\tau: G\to J$ of $G$ onto an ICC group $J$, lifts to a homomorphism $\rho: I\to J$ such that the following diagram commutes.
  $$
  \begin{tikzcd}
  {} & I \arrow{dr}{\rho} \\
  G \arrow{ur}{\phi} \arrow{rr}{\tau} && J
  \end{tikzcd}
  $$
\end{definition}

Now we can state the following proposition.

\begin{proposition}
\label{prop:hyper-fc-center}
  Let $G$ be a group and let $H\trianglelefteq G$ be the hyper-FC center of $G$. The quotient map $\phi: G\to G/H$ is the unique, up to isomorphism, universal ICC quotient of $G$.
\end{proposition}
\begin{proof}
  First, we show that $H$ is in the kernel of any ICC quotient $\tau: G\to J$. Let
  $$\{e\}=F_0\leqslant F_1\leqslant \cdots \leqslant F_\alpha \leqslant \cdots $$
  be the upper FC-series of $G$. If $H$ is not in the kernel of $\tau$, then there exists a minimum ordinal $\alpha$ such that $F_\alpha\not\subseteq \ker{\tau}$. Obviously $\alpha$ is not a limit ordinal. Let $h\in F_\alpha\setminus \ker{\tau}$. Since $hF_{\alpha-1}$ is FC in $G/F_{\alpha-1}$ and $F_{\alpha-1}\subseteq \ker{\tau}$ and $\tau$ is a surjective homomorphism, we get that $\tau(h)$ is FC in $J$. So, $\tau(h)$ is the identity in $J$, which is a contradiction. So, $H$ is in the kernel of any ICC quotient of $G$.
  
  Now, we show that the quotient map $\phi: G\to G/H$ is an ICC quotient. Let $\gamma$ be the FC-rank of $G$. So $H=F_\gamma$. Note that since $F_\gamma = F_{\gamma+1}$, we know that any non-identity element of $G/H = G/F_\gamma$ has infinitely many conjugates, which shows that $G/H$ is ICC.
  
  Thus $\phi: G\to G/H$ is a universal ICC quotient. Uniqueness follows from a standard fact about universal properties.
\end{proof}

An immediate corollary of the above result is that hyper-FC central groups are exactly those with no non-trivial ICC quotients. Theorem~\ref{thm: f.g. virtually nilpotent}, which we prove next, gives a third equivalent condition when the group is finitely generated.

\begin{proof}[Proof of Theorem~\ref{thm: f.g. virtually nilpotent}]
  The equivalence of (2) and (3) follows from the above corollary. We will show that (1) and (2) are equivalent. For that, we first show that the upper FC-series of a finitely generated hyper-FC central group stabilizes at some finite ordinal, i.e.\ the FC-rank is finite.
  
  \begin{claim}
  \label{claim:f.g. hyper-FC central}
    Let $G$ be a finitely generated hyper-FC central group. The upper FC-series of $G$ stabilizes at some $n\in\N$, i.e.\ its FC-rank is finite.
  \end{claim}
  \begin{proof}
    Let $S$ be a finite symmetric generating set for $G$. Let
    $$\{e\}=F_0\leqslant F_1\leqslant \cdots \leqslant F_\alpha \leqslant \cdots \leqslant F_\gamma = G$$
    be the upper FC-series of $G$. We need to show that for some $n\in \N$ we have $S\subseteq F_n$. For that, we will define a sequence $X_0,X_1,\ldots$ of finite subsets of $G$ with the following properties:
    \begin{enumerate}
      \item $X_0=S$.
      \item If $\alpha_i$ is the least ordinal with $X_i\subseteq F_{\alpha_i}$, then either $\alpha_i = \alpha_{i-1} = 0$ or $\alpha_i = \alpha_{i-1}-1$.
    \end{enumerate}
    Given such a sequence, if none of the $\alpha_i$'s are 0, then $\alpha_0,\alpha_1,\ldots$ is an infinite strictly decreasing sequence of ordinals, which is a contradiction. So, some $\alpha_i$ is 0. Let $n$ be the least index with $\alpha_n = 0$. Then $\alpha_0=n$. By the definition of $\alpha_0$, we get that $S=X_0\subseteq F_n$. But since $S$ generates $G$, we get that $G\subseteq F_n$. So the upper FC-series stabilizes at $n$.
    
    Now we define the sequence $X_0,X_1,\ldots$ and prove it has the properties we claimed. Let $X_0=S$. Assume that $X_0,\ldots,X_i$ are defined. We want to define $X_{i+1}$. If $\alpha_i = 0$, then simply let $X_{i+1}=X_i$. And if $\alpha_i\neq 0$, we define $X_{i+1}$ below. First, we make a few observations. 
    \begin{itemize}
      \item Note that since $\alpha_i$ is the least ordinal such that $F_{\alpha_i}$ contains the finite set $X_i$, we get that $\alpha_i$ is not a limit ordinal.
      \item Since $F_{\alpha_i}/F_{\alpha_i-1}$ is the FC-center of $G/F_{\alpha_i-1}$ and $X_i\subseteq F_{\alpha_i}$, we have that $x F_{\alpha_i-1}$ is FC in $G/F_{\alpha_i-1}$ for each $x\in X_i$.
      \item For each $x\in X_i$ and each conjugate of $x F_{\alpha_i-1}$, pick an element of $G$ in that conjugate, and let $Y_i$ be the union of $X_i$ and the collection of all the elements we chose. So, $X_i\subseteq Y_i\subseteq F_{\alpha_i}$ and $Y_i$ is finite.
    \end{itemize}
    Note that if $y\in Y_i$ and $g\in G$, then $A = (g^{-1}y g) F_{\alpha_i-1}$ is a conjugate of $x F_{\alpha_i-1}$ for some $x\in X_i$. Thus $A = z F_{\alpha_i-1}$ for some $z\in Y_i$, and so $z^{-1} (g^{-1} y g)\in F_{\alpha_i-1}$ for some $z\in Y_i$. Let
    $$X_{i+1} = \{ z^{-1} (g^{-1} y g) \ | \ z^{-1} (g^{-1} y g)\in F_{\alpha_i-1},\ g\in S,\ y,z\in Y_i\}.$$
    Note that $X_{i+1}$ is finite and $X_{i+1}\subseteq F_{\alpha_i-1}$. So, $\alpha_{i+1}\leq \alpha_i-1$. To show that $\alpha_{i+1} = \alpha_i-1$, we just need to show that $\alpha_i \leq \alpha_{i+1}+1$, i.e.\ $X_i\subseteq F_{\alpha_{i+1}+1}$, which is the same as showing that $x F_{\alpha_{i+1}}$ is FC in $G/F_{\alpha_{i+1}}$ for all $x\in X_i$. Since $X_i\subseteq Y_i$, it suffices to show that $y F_{\alpha_{i+1}}$ is FC in $G/F_{\alpha_{i+1}}$ for all $y\in Y_i$.\\
    Since $X_{i+1}\subseteq F_{\alpha_{i+1}}$, for any $y\in Y_i$ and $s\in S$, there exists a $z\in Y_i$ with $z^{-1} (s^{-1}ys)\in F_{\alpha_{i+1}}$. Since $S$ generates $G$ and $F_{\alpha_{i+1}}$ is normal in $G$, the same holds for any $g\in G$ replacing $s\in S$. Thus $Y_i F_{\alpha_{i+1}}\subset G/F_{\alpha_{i+1}}$ is closed under taking conjugates. Since $Y_i$ is finite, $y F_{\alpha_{i+1}}$ is thus FC in $G/F_{\alpha_{i+1}}$ for any $y\in Y_i$. This completes the proof.
  \end{proof}

  Now, we show that a finitely generated group has finite FC-rank if and only if it is virtually nilpotent. For $n\in\N$, denote the class of finitely generated hyper-FC central groups of FC-rank less than or equal to $n$ by $\mathcal{FC}_n$, and the class of finitely generated virtually nilpotent groups of rank less than or equal to $n$ by $\mathcal{VN}_n$.
  
  First we prove a useful lemma. 
  \begin{lemma}
  \label{lemma:f.g. fc}
    Let $G$ be a group, and $H$ be a finitely generated subgroup of the FC-center of $G$. The centralizer of $H$ in $G$, denoted by $\cent{H}$, has finite index in $G$.
  \end{lemma}
  \begin{proof}
    Let $\{h_1,...h_n\}$ be a set of generators for $H$. Note that for each $h_i$, since it is FC in $G$, its centralizer $\cent{h_i}$ has finite index. Thus, the intersection of the centralizers $\cap_{i=1}^{n} \cent{h_i}$, which is the same as $\cent{H},$ has finite index in $G$.
  \end{proof}
  
  \begin{claim}
  \label{claim:f.g. nilpotent fc}
    We have
    $$\mathcal{VN}_0 \subseteq \mathcal{FC}_1\subseteq \mathcal{VN}_1\subseteq \mathcal{FC}_2 \subseteq \cdots\subseteq \mathcal{FC}_n \subseteq \mathcal{VN}_n \subseteq \mathcal{FC}_{n+1} \subseteq \cdots.$$
  \end{claim}
  \begin{proof}
    First, we show that $\mathcal{VN}_{n-1}\subseteq \mathcal{FC}_n$ for any $n\in\N$. Let $G$ be a group in $\mathcal{VN}_{n-1}$ for $n\in\N$. Let $N\trianglelefteq G$ be a finite index normal subgroup with the upper central series
    $$\{e\}=Z_0 \leqslant Z_1 \leqslant \cdots \leqslant Z_m = N,$$
    where $m\leq n-1$. Since $Z_1$ is the center of a normal subgroup of $G$, we get that $Z_1$ is normal in $G$. Similarly, we can show that each $Z_k$ is normal in $G$. Since $Z_k/Z_{k-1}$ is in the center of $N/Z_{k-1}$, we get that $N/Z_{k-1} \leqslant C_{G/Z_{k-1}}(z Z_{k-1})$ for any $z\in Z_k$, which means that $C_{G/Z_{k-1}}(z Z_{k-1})$ is of finite index in $G/Z_{k-1}$ for any $z\in Z_k$. So, $Z_k/Z_{k-1}$ is in the FC-center of $G/Z_{k-1}$. Obviously, since $G/N$ is finite, we have that $G/N$ is FC. So, we have that
    $$\{e\}=Z_0 \leqslant Z_1 \leqslant \cdots \leqslant Z_m = N \leqslant G$$
    is an FC-series for $G$ with length $m+1\leq n$. So, $G$ belongs to $\mathcal{FC}_n$.

    Now, by induction on $n\in \N$ we show that $\mathcal{FC}_n\subseteq \mathcal{VN}_n$. Let $G$ be a group that belongs to $\mathcal{FC}_1$. By Lemma~\ref{lemma:f.g. fc}, the center of $G$ has finite index in $G$. So, $G$ is virtually abelian, which means that $G$ belongs to $\mathcal{VN}_1$. Thus $\mathcal{FC}_1 \subseteq \mathcal{VN}_1$.
    
    Let $G$ be a group that belongs to $\mathcal{FC}_n$ for $n\geq 2$. Let
    $$\{e\} = F_0 \leqslant F_1 \leqslant \cdots \leqslant F_m = G$$
    be the upper FC-series of $G$, where $m\leq n$. Since $G/F_1$ is in $\mathcal{FC}_{m-1}$, by the induction hypothesis we know that $G/F_1$ is virtually nilpotent of rank at most $m-1$. So, there is a normal subgroup $N\trianglelefteq G$ with finite index such that $F_1\leqslant N$ and $N/F_1$ is nilpotent of rank at most $m-1$. We can make the following observations:
    \begin{itemize}
      \item Since $N$ has finite index in a finitely generated group, $N$ is finitely generated. Let $S$ be a finite symmetric set of generators for $N$.
      \item Let $N = \Gamma_0 \trianglerighteq \Gamma_1 \trianglerighteq \cdots \trianglerighteq \Gamma_{m-1}$ be the first $m$ subgroups in the lower central series of $N$. Since $N/F_1$ is nilpotent of rank at most $m-1$, we know that $\Gamma_{m-1}\leqslant F_1$. So, $\Gamma_{m-1}$ is FC.
      \item It is easy to see that $\Gamma_{m-1}$ is the least normal subgroup of $N$ that contains all the $(m-1)$-fold commutators $[s_1,s_2,\ldots,s_{m-1}]$, where $s_i$'s are elements of $S$. Note that 1) since $\Gamma_{m-1}$ is FC, we know that each of $[s_1,s_2,\ldots,s_{m-1}]$ has finitely many conjugates in $N$, and 2) since $S$ is finite, we have finitely many elements of the form $[s_1,s_2,\ldots,s_{m-1}]$. So, $\Gamma_{m-1}$ is finitely generated.
    \end{itemize}
    From the last two observations we know that $\Gamma_{m-1}$ is a finitely generated FC subgroup of $N$. By Lemma~\ref{lemma:f.g. fc}, we know that $C_N(\Gamma_{m-1})$ has finite index in $N$. Obviously, $C_N(\Gamma_{m-1})$ has a normal subgroup $M$ with finite index in $N$. It is clear that $Z=\Gamma_{m-1}\cap M$ is in the center of $M$. Since $N$ has finite index in $G$, we get that $M$ also has finite index in $G$.
    
    By the second isomorphism theorem for groups, we have that
    $$M/Z = M/(\Gamma_{m-1}\cap M)\cong (M\Gamma_{m-1})/\Gamma_{m-1} \leqslant N/\Gamma_{m-1}.$$
    But we know that $N/\Gamma_{m-1}$ is nilpotent with rank at most $m-1$. So, $M/Z$ is nilpotent with rank at most $m-1$. Also, $Z$ is in the center of $M$. Hence, $M$ is nilpotent with rank at most $m\leq n$. So, $G$ is virtually nilpotent with rank at most $n$.
  \end{proof}
  
  Now, we can show that (1) and (2) are equivalent. Let $G$ be a finitely generated group.\\
  If $G$ is a virtually nilpotent group of rank $n$, then by Claim~\ref{claim:f.g. nilpotent fc} we know that it is FC with FC-rank at most $n+1$.\\
  If, on the other hand, $G$ is hyper-FC central, then by Claim~\ref{claim:f.g. hyper-FC central} we know that its FC-rank is finite, say $n\in \N$. So, by Claim~\ref{claim:f.g. nilpotent fc} we know that it is virtually nilpotent of rank at most $n$.
\end{proof}

\bibliography{main.bib}
\end{document}